\documentclass[11pt,oneside]{amsart}
\usepackage[latin9]{inputenc}
\usepackage{enumitem}
\usepackage{amsthm}
\usepackage{amssymb}

\makeatletter
\numberwithin{equation}{section}
\numberwithin{figure}{section}
\theoremstyle{plain}
\newtheorem{thm}{\protect\theoremname}[section]
\theoremstyle{definition}
\newtheorem{defn}[thm]{\protect\definitionname}
\theoremstyle{plain}
\newtheorem{fact}[thm]{\protect\factname}
\ifx\proof\undefined
\newenvironment{proof}[1][\protect\proofname]{\par
\normalfont\topsep6\p@\@plus6\p@\relax
\trivlist
\itemindent\parindent
\item[\hskip\labelsep\scshape #1]\ignorespaces
}{%
\endtrivlist\@endpefalse
}
\providecommand{\proofname}{Proof}
\fi
\theoremstyle{remark}
\newtheorem{rem}[thm]{\protect\remarkname}
 \theoremstyle{remark}
 \newtheorem{notation}[thm]{\protect\notationname}
\theoremstyle{plain}
\newtheorem{prop}[thm]{\protect\propositionname}
\theoremstyle{remark}
\newtheorem{claim}[thm]{\protect\claimname}
 \theoremstyle{lemma}
 \newtheorem{mainlemma}[thm]{\protect\mainlemmaname}
 \newlist{casenv}{enumerate}{4}
 \setlist[casenv]{leftmargin=*,align=left,widest={iiii}}
 \setlist[casenv,1]{label={{\itshape\ \casename} \arabic*.},ref=\arabic*}
 \setlist[casenv,2]{label={{\itshape\ \casename} \roman*.},ref=\roman*}
 \setlist[casenv,3]{label={{\itshape\ \casename\ \alph*.}},ref=\alph*}
 \setlist[casenv,4]{label={{\itshape\ \casename} \arabic*.},ref=\arabic*}
\theoremstyle{plain}
\newtheorem{cor}[thm]{\protect\corollaryname}
\theoremstyle{definition}
\newtheorem{example}[thm]{\protect\examplename}
\theoremstyle{plain}
\newtheorem{lem}[thm]{\protect\lemmaname}

\usepackage{amsfonts}
\usepackage{nicefrac}
\usepackage{amscd}
\usepackage{a4wide}
\usepackage{url}
\newcommand{\generic}{generic}

\makeatother

 \providecommand{\casename}{Case}
 \providecommand{\mainlemmaname}{Main Lemma}
 \providecommand{\notationname}{Notation}
\providecommand{\claimname}{Claim}
\providecommand{\corollaryname}{Corollary}
\providecommand{\definitionname}{Definition}
\providecommand{\examplename}{Example}
\providecommand{\factname}{Fact}
\providecommand{\lemmaname}{Lemma}
\providecommand{\propositionname}{Proposition}
\providecommand{\remarkname}{Remark}
\providecommand{\theoremname}{Theorem}

\begin{document}
\global\long\def\p{\mathbf{p}}
\global\long\def\q{\mathbf{q}}
\global\long\def\C{\mathfrak{C}}
\global\long\def\SS{\mathcal{P}}
 \global\long\def\pr{\operatorname{pr}}
\global\long\def\image{\operatorname{im}}
\global\long\def\otp{\operatorname{otp}}
\global\long\def\dec{\operatorname{dec}}
\global\long\def\suc{\operatorname{suc}}
\global\long\def\pre{\operatorname{pre}}
\global\long\def\qe{\operatorname{qf}}
 \global\long\def\ind{\operatorname{ind}}
\global\long\def\Nind{\operatorname{Nind}}
\global\long\def\lev{\operatorname{lev}}
\global\long\def\Suc{\operatorname{Suc}}
\global\long\def\HNind{\operatorname{HNind}}
\global\long\def\minb{{\lim}}
\global\long\def\concat{\frown}
\global\long\def\cl{\operatorname{cl}}
\global\long\def\tp{\operatorname{tp}}
\global\long\def\id{\operatorname{id}}
\global\long\def\cons{\left(\star\right)}
\global\long\def\qf{\operatorname{qf}}
\global\long\def\ai{\operatorname{ai}}
\global\long\def\dtp{\operatorname{dtp}}
\global\long\def\acl{\operatorname{acl}}
\global\long\def\nb{\operatorname{nb}}
\global\long\def\limb{{\lim}}
\global\long\def\leftexp#1#2{{\vphantom{#2}}^{#1}{#2}}
\global\long\def\intr{\operatorname{interval}}
\global\long\def\atom{\emph{at}}
\global\long\def\I{\mathfrak{I}}
\global\long\def\uf{\operatorname{uf}}
\global\long\def\ded{\operatorname{ded}}
\global\long\def\Ded{\operatorname{Ded}}
\global\long\def\Df{\operatorname{Df}}
\global\long\def\Th{\operatorname{Th}}
\global\long\def\eq{\operatorname{eq}}
\global\long\def\Aut{\operatorname{Aut}}
\global\long\def\ac{ac}
\global\long\def\DfOne{\operatorname{df}_{\operatorname{iso}}}
\global\long\def\modp#1{\pmod#1}
\global\long\def\sequence#1#2{\left\langle #1\,\middle|\,#2\right\rangle }
\global\long\def\set#1#2{\left\{  #1\,\middle|\,#2\right\}  }
\global\long\def\Diag{\operatorname{Diag}}
\global\long\def\Nn{\mathbb{N}}
\global\long\def\mathrela#1{\mathrel{#1}}
\global\long\def\twiddle{\mathord{\sim}}
\global\long\def\mathordi#1{\mathord{#1}}
\global\long\def\Qq{\mathbb{Q}}
\global\long\def\dense{\operatorname{dense}}
 \global\long\def\cof{\operatorname{cof}}
\global\long\def\tr{\operatorname{tr}}
\global\long\def\treeexp#1#2{#1^{\left\langle #2\right\rangle _{\tr}}}
\global\long\def\x{\times}
\global\long\def\forces{\Vdash}
\global\long\def\Vv{\mathbb{V}}
\global\long\def\Uu{\mathbb{U}}
\global\long\def\tauname{\dot{\tau}}
\global\long\def\ScottPsi{\Psi}
\global\long\def\cont{2^{\aleph_{0}}}
\global\long\def\MA#1{{MA}_{#1}}
\global\long\def\rank#1#2{R_{#1}\left(#2\right)}
\global\long\def\cal#1{\mathcal{#1}}

\def\Ind#1#2{#1\setbox0=\hbox{$#1x$}\kern\wd0\hbox to 0pt{\hss$#1\mid$\hss} \lower.9\ht0\hbox to 0pt{\hss$#1\smile$\hss}\kern\wd0} 
\def\Notind#1#2{#1\setbox0=\hbox{$#1x$}\kern\wd0\hbox to 0pt{\mathchardef \nn="3236\hss$#1\nn$\kern1.4\wd0\hss}\hbox to 0pt{\hss$#1\mid$\hss}\lower.9\ht0 \hbox to 0pt{\hss$#1\smile$\hss}\kern\wd0} 
\def\nind{\mathop{\mathpalette\Notind{}}}

\global\long\def\ind{\mathop{\mathpalette\Ind{}}}
\global\long\def\opp{\operatorname{opp}}
 \global\long\def\nind{\mathop{\mathpalette\Notind{}}}
\global\long\def\average#1#2#3{Av_{#3}\left(#1/#2\right)}
\global\long\def\Ff{\mathfrak{F}}
\global\long\def\mx#1{Mx_{#1}}
\global\long\def\maps{\mathfrak{L}}

\global\long\def\Esat{E_{\mbox{sat}}}
\global\long\def\Ebnf{E_{\mbox{rep}}}
\global\long\def\Ecom{E_{\mbox{com}}}
\global\long\def\BtypesA{S_{\Bb}^{x}\left(A\right)}

\global\long\def\init{\trianglelefteq}
\global\long\def\fini{\trianglerighteq}
\global\long\def\Bb{\cal B}
\global\long\def\Lim{\operatorname{Lim}}
\global\long\def\Succ{\operatorname{Succ}}

\global\long\def\SquareClass{\cal M}
\global\long\def\leqstar{\leq_{*}}
\global\long\def\average#1#2#3{Av_{#3}\left(#1/#2\right)}
\global\long\def\cut#1{\mathfrak{#1}}

\global\long\def\OurSequence{\mathcal{I}}

\title{Exact saturation in simple and NIP theories}

\author{Itay Kaplan, Saharon Shelah and Pierre Simon}

\thanks{The first author would like to thank the Israel Science foundation
for partial support of this research (Grant no. 1533/14). }

\thanks{The research leading to these results has received funding from the
European Research Council, ERC Grant Agreement n. 338821. No. F1473
on the third author's list of publications.}

\thanks{Partially supported by ValCoMo (ANR-13-BS01-0006).}

\address{Itay Kaplan \\
The Hebrew University of Jerusalem\\
Einstein Institute of Mathematics \\
Edmond J. Safra Campus, Givat Ram\\
Jerusalem 91904, Israel}

\email{kaplan@math.huji.ac.il}

\urladdr{https://sites.google.com/site/itay80/ }

\address{Saharon Shelah\\
The Hebrew University of Jerusalem\\
Einstein Institute of Mathematics \\
Edmond J. Safra Campus, Givat Ram\\
Jerusalem 91904, Israel}

\address{Saharon Shelah \\
Department of Mathematics\\
Hill Center-Busch Campus\\
Rutgers, The State University of New Jersey\\
110 Frelinghuysen Road\\
Piscataway, NJ 08854-8019 USA}

\email{shelah@math.huji.ac.il}

\urladdr{http://shelah.logic.at/}

\address{Pierre Simon\\
Institut Camille Jordan\\
Université Claude Bernard - Lyon 1\\
43 boulevard du 11 novembre 1918\\
69622 Villeurbanne Cedex, France}

\email{simon@math.univ-lyon1.fr}

\urladdr{http://www.normalesup.org/\textasciitilde{}simon/}

\begin{abstract}
A theory $T$ is said to have exact saturation at a singular cardinal
$\kappa$ if it has a $\kappa$-saturated model which is not $\kappa^{+}$-saturated.
We show, under some set-theoretic assumptions, that any simple theory
has exact saturation. Also, an NIP theory has exact saturation if
and only if it is not distal. This gives a new characterization of
distality.
\end{abstract}

\maketitle

\section{Introduction}

A first order theory $T$ has exact saturation at $\kappa$ if it
has a $\kappa$-saturated model which is not $\kappa^{+}$-saturated.
When $\kappa>|T|$ is regular, then any theory has exact saturation
at $\kappa$ (see Fact \ref{fact:Unstable - regular}), hence we are
only interested in the case $\kappa$ singular.

Possibly adding set-theoretic assumptions, we expect that for a given
theory $T$, having exact saturation at a singular cardinal $\kappa$
does not depend on $\kappa$, and that this property is an interesting
dividing line within first order theories. We indeed show this for
stable, simple and NIP theories.

The second author has shown previously \cite[IV, Lemma 2.18]{Sh:c}
that stable theories have exact saturation at any $\kappa$. Since
this is not stated exactly in this form there, and also for completeness,
we added a proof here (see Theorem \ref{thm:stable}). He also showed
that an NIP theory with an infinite indiscernible set has exact saturation
at any singular $\kappa$ with $2^{\kappa}=\kappa^{+}$ (\cite[Claim 2.26]{Sh900}). 

We establish here the precise dividing line for NIP theories: with
the same assumptions on $\kappa$, an NIP theory has exact saturation
at $\kappa$ if and only if it is not distal. This gives a new characterization
of distality within NIP theories, and allows an answer to Question
2.30 from \cite{Sh900}. See Corollary \ref{cor:Answer to 900}.

We also generalize the result on stable theories to simple theories:
let $T$ be simple and assume that $\kappa$ is singular of cofinality
greater than $|T|$, $2^{\kappa}=\kappa^{+}$ and $\square_{\kappa}$
holds, then $T$ has exact saturation at $\kappa$.

\section{Definitions and first results}
\begin{defn}
\label{def:Exact saturation}Suppose $T$ is a first order theory
and $\kappa$ is a cardinal. We say that $T$ \emph{has exact saturation
at $\kappa$} if $T$ has a $\kappa$-saturated model $M$ which is
not $\kappa^{+}$-saturated.
\end{defn}

We will use the following notion throughout the paper.
\begin{defn}
\label{def:D-types}Let $T$ be any complete theory. Suppose that
$D$ is a collection of finitary types over some set $A$. A set $B$
is a \emph{$D$-set }if for every finite tuple $b$ from $B$, $\tp\left(b/A\right)\in D$.
For A $D$-set $B\supseteq A$, a type $p\in S^{<\omega}\left(B\right)$
is called a\emph{ $D$-type }if $Bd$ is a $D$-set for some (any)
$d\models p$. A \emph{$D$-model} is a model of $T$ which is a $D$-set. 
\end{defn}

\subsection{Stable theories}

Suppose $T$ is a stable theory. This part is not new, but it is short,
and we keep it for completeness. 
\begin{fact}
\label{fact:stable-isolation}\cite[IV, Lemma 2.18]{Sh:c}Suppose
$p\left(x\right)$ is a partial type over $B\subseteq A$. Then there
is $A_{0}\subseteq A$ of size $\leq\left|T\right|$, and an extension
$q\supseteq p$ over $BA_{0}$ which isolates a complete type over
$A$.\end{fact}
\begin{proof}
Enumerate the formulas $\sequence{\varphi_{i}\left(x,y_{i}\right)}{i<\left|T\right|}$.
Construct an increasing continuous sequence of types $\sequence{p_{i}}{i<\left|T\right|}$
where $p_{0}=p$ such that $p_{i+1}$ isolates a complete $\varphi_{i}$-type
over $A$ and $\left|p_{i+1}\backslash p_{i}\right|=1$. To find $p_{i+1}$,
let $\psi\left(x\right)$ be a formula over $A$ with minimal $R_{2,\varphi_{i}}$-rank
consistent with $p_{i}$ (which exists by stability), and let $p_{i+1}=p_{i}\cup\left\{ \psi\right\} $.
Finally, let $q=\bigcup_{i<\left|T\right|}p_{i}$. \end{proof}
\begin{thm}
\label{thm:stable}Assume $T$ is stable. Then for all $\kappa>\left|T\right|$,
$T$ has exact saturation at $\kappa$. \end{thm}
\begin{proof}
Let $I$ be an indiscernible set of cardinality $\kappa$. Let $D$
be the collection of finitary types $p$ over $I$ such that for some
$I_{0}\subseteq I$ of cardinality $<\kappa$, $p|_{I_{0}}\models p$. 

Suppose that $p\left(x\right)=\tp\left(c/I\right)\in D$, as witnessed
by $I_{0}$. By stability, there is some $I'\subseteq I$ of size
$\leq\left|T\right|$ such that $I\backslash I'$ is indiscernible
over $cI'$ (let $I'$ be the set of parameters appearing in the formulas
defining $\tp\left(c/I\right)$ over $I$, or see Fact \ref{fact:(Shrinking)}
below). Let $I_{0}'=I'\cup I''$ where $I''\subseteq I$ is countable
disjoint from $I'$. Then an easy argument\footnote{Suppose $\varphi\left(x,a\right)\in p$. Let $\psi\left(x,b\right)\in p$,
$b\in I_{0}$ be such that $\psi\left(x,b\right)\models\varphi\left(x,a\right)$.
Apply an automorphism to move $b\cup a$ over $I'$ to $I_{0}'$ to
get $b'\cup a'$. We get that $\psi\left(x,b'\right)\models\varphi\left(x,a'\right)$.
By indiscernibility over $cI'$ and as $\psi\left(x,b\right)\in p$,
so is $\psi\left(x,b'\right)\in p$, so $c\models\varphi\left(x,a'\right)$,
but then by the same reason $c\models\varphi\left(x,a\right)$. } gives us that $p|_{I_{0}'}\models p$. This shows that we can always
assume that $I_{0}$ has size $\leq\left|T\right|$. 

By Fact \ref{fact:stable-isolation}, we can construct a $\kappa$-saturated
$D$-model $M$ containing $I$ of cardinality $2^{\kappa}$. It is
enough to show that given $D$-sets $A\subseteq B$ where $\left|A\right|<\kappa$,
and some type $p\in S\left(A\right)$, there is some realization $a\models p$
such that $aB$ is a $D$-set. We may assume that $I\subseteq B$.
By Fact \ref{fact:stable-isolation}, there is some $A_{0}\subseteq B$
such that $\left|A_{0}\right|\leq\left|T\right|$ and a type $p_{0}\supseteq p$
over $A_{0}A$ such that $p_{0}$ isolates a complete type over $B$.
Let $a\models p_{0}$. Then $aB$ is a $D$-set: for a finite tuple
$c$ from $B$, let $I_{0}$ be such that $\tp\left(A_{0}Ac/I_{0}\right)\vdash\tp\left(A_{0}Ac/I\right)$
and $\left|I_{0}\right|\leq\left|A_{0}\right|+\left|A\right|+\left|T\right|<\kappa$.
Then $\tp\left(ac/I_{0}\right)\vdash\tp\left(ac/I\right)$. 

Now note that $\average II{}$ which is a type of a new element in
$I$ over $I$ is not a $D$-type, so $M$ is not $\kappa^{+}$-saturated. 
\end{proof}

\subsection{Unstable theories}

Recall that a type $p\left(x\right)\in S\left(M\right)$ is called
\emph{invariant} over $A\subseteq M$ if it does not split over $A$:
if $a,b\in M$ are such that $a\equiv_{A}b$ then for any formula
$\varphi\left(x,y\right)$, $\varphi\left(x,a\right)\in p\Leftrightarrow\varphi\left(x,b\right)\in p$. 
\begin{fact}
\label{fact:Unstable - regular}If $T$ is not stable then $T$ has
exact saturation at any regular $\left|T\right|<\kappa$.
\begin{proof}
Let $M_{0}\models T$ be of size $\left|T\right|$. For $i\leq\kappa$,
define a continuous increasing sequence of models $M_{i}$ where $\left|M_{i+1}\right|=2^{\left|M_{i}\right|}$
and $M_{i+1}$ is $\left|M_{i}\right|^{+}$-saturated. Hence $M_{\kappa}$
is $\kappa$-saturated and $\left|M_{\kappa}\right|=\beth_{\kappa}\left(\left|T\right|\right)$.

As $T$ is unstable, $\left|S\left(M_{\kappa}\right)\right|>\beth_{\kappa}\left(\left|T\right|\right)$.
However, the number of types over $M_{\kappa}$ invariant over $M_{i}$
is $\leq2^{2^{\left|M_{i}\right|}}\leq\beth_{\kappa}\left(\left|T\right|\right)$,
there is $p\left(x\right)\in S\left(M_{\kappa}\right)$ which splits
over every $M_{i}$. Hence for each $i<\kappa$, there is some formula
$\varphi_{i}\left(x,y\right)$ and some $a_{i},b_{i}\in M_{\kappa}$
such that $a_{i}\equiv_{M_{i}}b_{i}$ and $\varphi_{i}\left(x,a_{i}\right)\land\neg\varphi_{i}\left(x,b_{i}\right)\in p$.
Let $q\left(x\right)$ be $\set{\varphi_{i}\left(x,a_{i}\right)\land\neg\varphi_{i}\left(x,b_{i}\right)}{i<\kappa}$.
Then $q$ is not realized in $M_{\kappa}$. 
\end{proof}
\end{fact}

\section{Simple theories}

In the following definition, $\Lim\left(A\right)$ for a set of ordinals
$A$ is the set of ordinals $\delta\in A$ which are limits of ordinals
in $A$. 
\begin{defn}
\label{def:square}(\emph{Jensen's Square principle}, \cite[Page 443]{JechSetTheory})
Let $\kappa$ be an uncountable cardinal; $\square_{\kappa}$ (square-$\kappa$)
is the following condition:

There exists a sequence $\sequence{C_{\alpha}}{\alpha\in\Lim\left(\kappa^{+}\right)}$
such that:
\begin{enumerate}
\item $C_{\alpha}$ is a closed unbounded subset of $\alpha$.
\item If $\beta\in\Lim\left(C_{\alpha}\right)$ then $C_{\beta}=C_{\alpha}\cap\beta$.
\item If $\cof\left(\alpha\right)<\kappa$, then $\left|C_{\alpha}\right|<\kappa$. 
\end{enumerate}
\end{defn}
\begin{rem}
\label{rem:Square'}Suppose that $\sequence{C_{\alpha}}{\alpha\in\Lim\left(\kappa^{+}\right)}$
witness $\square_{\kappa}$. Let $C_{\alpha}'=\Lim\left(C_{\alpha}\right)$.
Then the following holds for $\alpha\in\Lim\left(\kappa^{+}\right)$. 
\begin{enumerate}
\item If $C_{\alpha}'\neq\emptyset$, then either $\sup\left(C_{\alpha}'\right)=\alpha$,
or $C_{\alpha}'$ has a last element $<\alpha$ in which case $\cof\left(\alpha\right)=\omega$.

\item $C_{\alpha}'\subseteq\Lim\left(\alpha\right)$ and for all $\beta\in C_{\alpha}'$,
$C_{\alpha}'\cap\beta=C_{\beta}'$. 
\item If $\cof\left(\alpha\right)<\kappa$, then $\left|C_{\alpha}'\right|<\kappa$.
\end{enumerate}
\end{rem}
\begin{thm}
\label{thm:Main-Simple theories}Suppose that $T$ is simple, $\mu$
is singular with $\left|T\right|<\kappa=\cof\left(\mu\right)$, $\mu^{+}=2^{\mu}$
and $\square_{\mu}$ holds. Then $T$ has exact saturation at $\mu$. 
\end{thm}
The square assumption will only be used in the  end of the proof. 

Towards the proof, let us first fix an increasing continuous sequence
$\sequence{\lambda_{i}}{i<\kappa}$ of cardinals whose limit is $\mu$
such that $\lambda_{i+1}$ is regular for all $i<\kappa$ and such
that $\lambda_{0}>\kappa$. 

If $T$ is stable, then we already know that $T$ has exact saturation
at $\mu$ by Theorem \ref{thm:stable}. So assume that $T$ is not
stable.

As it is simple, by e.g., \cite[Exercises 8.2.5, 8.2.6]{TentZiegler},
it has the independence property. Let $\varphi\left(x,y\right)$ witness
this. 
\begin{notation}
\label{notation: sum of orders}For a sequence of linear orders $\sequence{\left(X_{i}<_{i}\right)}{i\in I}$
where $I$ is linearly ordered by $<$, let $\sum_{i\in I}X_{i}$
be the linear order whose set of elements is $\bigcup\set{X_{i}\x\left\{ i\right\} }{i\in I}$
ordered by $\left(x,i\right)<\left(y,j\right)$ iff $i<j$ or $i=j$
and $x<_{i}y$. 
\end{notation}
Let $\Succ\left(\kappa\right)=\kappa\backslash\Lim\left(\kappa\right)$.
For $i\in\Succ\left(\kappa\right)$, let $I_{i}$ be the linear order
$\lambda_{i}$, and let $I=\sum_{i\in\Succ\left(\kappa\right)}I_{i}$.
Let $\sequence{a_{i}}{i\in I}$ be an indiscernible sequence witnessing
that $\varphi$ has the independence property. I.e., for every subset
$s\subseteq I$, there is some $b_{s}$ such that $\varphi\left(b_{s},a_{i}\right)$
holds iff $i\in s.$ Abusing notation, we will write $I_{i}=\sequence{a_{j}}{j\in I_{i}}$
and similarly for $I$. 
\begin{defn}
\label{def:good}For $i\in\Succ\left(\kappa\right)$, let $D_{i}$
be the collection of finitary types $p\in S^{<\omega}\left(I_{i}\right)$
such that for any finite $s\subseteq I_{i}$ there is some $\alpha<\lambda_{i}$
such that $I_{i}^{\geq\alpha}$ (i.e., $I_{i}\upharpoonright\left[\alpha,\lambda_{i}\right)$)
is indiscernible over $s\cup d$ for some (any) $d\models p$. \end{defn}
\begin{rem}
\label{rem:Equivalent definition of D_i}Note that since $\lambda_{i}$
is regular when $i$ is a successor, a set $A$ is a $D_{i}$-set
iff for any subset $C\subseteq A$, $\left|C\right|<\lambda_{i}$,
there is some $\alpha<\lambda_{i}$ such that $I_{i}^{\geq\alpha}$
is indiscernible over $C\cup I_{i}^{<\alpha}$. Indeed, given $C$,
for every finite set $s\subseteq C$, let $\alpha_{s}$ be as in Definition
\ref{def:good} (for $s=\emptyset$). Let $\alpha_{0}=\sup\set{\alpha_{s}}{s\subseteq C\mbox{ finite}}$.
Let $\alpha_{0}<\alpha_{1}$ be defined similarly for $C\cup I_{i}^{<\alpha_{0}}$.
Continue and finally put $\alpha=\sup\set{\alpha_{n}}{n<\omega}$. \end{rem}
\begin{defn}
\label{def:The square class}Let $\SquareClass$ be the class of sequences
$\bar{A}=\sequence{A_{i}}{i<\kappa}$ such that:
\begin{itemize}
\item For some $i_{0}\in\Succ\left(\kappa\right)$, for all $i_{0}\leq i\in\Succ\left(\kappa\right)$,
$I_{i}\subseteq A_{i}$, and for all $i<i_{0}$, $A_{i}=\emptyset$;
$\sequence{A_{i}}{i<\kappa}$ is increasing and continuous and $\left|A_{i}\right|\leq\lambda_{i}$
for all $i\in\Succ\left(\kappa\right)$.
\item For all $i\in\Succ\left(\kappa\right)$, $A_{i}$ is a $D_{i}$-set. 
\end{itemize}
\end{defn}

\begin{defn}
\label{def:order}For $\bar{A},\bar{B}\in\SquareClass$, write $\bar{A}\leq_{i}\bar{B}$
for: for all $i\leq j<\kappa$, $A_{j}\subseteq B_{j}$;$\bar{A}\leq\bar{B}$
for: $\bar{A}\leq_{0}\bar{B}$; and $\bar{A}\leqstar\bar{B}$ for:
there is some $i<\kappa$ such that $\bar{A}\leq_{i}\bar{B}$. \end{defn}
\begin{prop}
\label{prop:Model}Given $\bar{A}\in\SquareClass$, there is $\bar{A}\leq\bar{B}\in\SquareClass$
such that for all $i\in\Succ\left(\kappa\right)$, $B_{i}$ is either
$\emptyset$ or a model of $T$. \end{prop}
\begin{proof}
For simplicity assume that $i_{0}=0$ in Definition \ref{def:The square class}.
It is enough to prove the following.
\begin{claim}
Let $i\in\Succ\left(\kappa\right)$. Suppose that $\varphi\left(x\right)$
is a formula over $A_{i}$. Then there is some $\bar{A}\leq\bar{B}\in\SquareClass$
such that $B_{i}$ realizes $\varphi$. 

\end{claim}
\begin{proof}
\renewcommand{\qedsymbol}{}Suppose $\varphi=\varphi\left(x,c\right)$.
Let $\alpha<\lambda_{i}$ be such that $I_{i}^{\geq\alpha}$ is indiscernible
over $A_{i-1}c\cup I_{i}^{<\alpha}$ (see Remark \ref{rem:Equivalent definition of D_i}).
Let $d\models\varphi$. by Ramsey and compactness, there is some sequence
$J$ with the same order type and $EM$-type as $I_{i}^{\geq\alpha}$
over $cdA_{i-1}I_{i}^{<\alpha}$ which is indiscernible over $cdA_{i-1}I_{i}^{<\alpha}$.
Hence $J\equiv_{cA_{i-1}I_{i}^{<\alpha}}I_{i}^{\geq\alpha}$, so apply
an automorphism of $\C$ to move $J$ to $I_{i}^{\geq\alpha}$ over
$cA_{i-1}I_{i}^{<\alpha}$ and let $d'$ be the image of $d$. We
get that $I_{i}^{\geq\alpha}$ is indiscernible over $cd'A_{i-1}I_{i}^{<\alpha}$
and still $d'\models\varphi$, and even $d'\equiv_{A_{i-1}}d$ (this
is not important here, but will be later). Let $p=\tp\left(d'/cA_{i-1}I_{i}\right)$. 

($\star$) Suppose now that we are in a general situation, where we
have some type $p_{1}\in S\left(A_{j}BI_{j+1}\right)$ where $B\subseteq A_{j+1}$
is of cardinality $<\lambda_{j+1}$ and there is some $\alpha<\lambda_{j+1}$
such that for any $d\models p_{1}$, $I_{j+1}^{\geq\alpha}$ is indiscernible
over $dA_{j}BI_{j+1}^{<\alpha}$, and suppose $e\in A_{j+1}$. Then
there is some $\lambda_{j+1}>\beta>\alpha$ such that $I_{j+1}^{\geq\beta}$
is indiscernible over $A_{j}BI_{j+1}^{<\beta}e$. By Ramsey and compactness,
there is some indiscernible sequence $J$ with the same $EM$ and
order type as $I_{j+1}^{\geq\beta}$ over $A_{j}BI_{j+1}^{<\beta}ed$
which is indiscernible over $A_{j}BI_{j+1}^{<\beta}ed$ for some fixed
$d\models p_{1}$. Then $J\equiv_{A_{j}BeI_{j+1}^{<\beta}}I_{j+1}^{\geq\beta}$
and $J\equiv_{A_{j}BdI_{j+1}^{<\beta}}I_{j+1}^{\geq\beta}$. Hence
applying an automorphism fixing $A_{j}BI_{j+1}^{<\beta}e$ which maps
$J$ to $I_{j+1}^{\geq\beta}$, we move $d$ to some $d'$ which still
realizes $p_{1}$ but now $I_{j+1}^{\geq\beta}$ is indiscernible
over $A_{j}BI_{j+1}^{\geq\beta}ed'$. Let $p_{2}=\tp\left(d'/A_{j}BI_{j+1}e\right)\supseteq p_{1}$. 

Using ($\star$) iteratively, taking unions at limit stages, starting
with $p$ we can find some $\varphi\in p_{i}\in S\left(A_{i}\right)$
such that for every $d\models p_{i}$, $A_{i}d$ is a $D_{i}$-set.

We construct an increasing continuous sequence of types, $\sequence{p_{j}}{i\leq j<\kappa}$,
$p_{j}\in S\left(A_{j}\right)$ such that for $j\in\Succ\left(\kappa\right)$,
if $e\models p_{j}$ then $A_{j}e$ is a $D_{j}$-set. We then let
$p=\bigcup_{i\leq j<\kappa}p_{j}$, $d\models p$ and define $B_{j}=A_{j}\cup\left\{ d\right\} $
for all $i\leq j<\kappa$. 

The construction of $p_{j+1}$ uses a weak version of ($\star$) in
the first step (keeping only the type over $A_{j}$, as we did in
the beginning), and then ($\star$) as in the construction of $p_{i}$. 
\end{proof}
\end{proof}
\begin{mainlemma}
\label{lem:Main Lemma simple}Suppose that $\sequence{A_{i}}{i<\kappa}\in\SquareClass$,
and $C\subseteq\bigcup_{i<\kappa}A_{i}$ is such that $\left|C\right|<\mu$.
Let $p\in S\left(C\right)$. Then there is some $\bar{A}\leq\bar{B}\in\SquareClass$
which contains a realization of $p$.\end{mainlemma}
\begin{proof}
Here we use the simplicity of $T$. 

First, by Proposition \ref{prop:Model}, we may assume that for $i\in\Succ\left(\kappa\right)$,
$A_{i}=\emptyset$ or is a model. 

We may assume that there is some $E\subseteq C$ of size $\leq\left|T\right|$
such that $p$ does not fork over $E$ and moreover if $q$ is a type
extending $p$ then $q$ does not fork over $E$. We get this by trying
to construct an increasing continuous sequence $\sequence{\left(p_{\alpha},E_{\alpha}\right)}{\alpha<\left|T\right|^{+}}$
of subsets $E_{\alpha}\subseteq\bigcup_{i<\kappa}A_{i}$ of cardinality
$\leq\left|T\right|$, and complete types $p_{\alpha}$ over $E_{\alpha}\cup C$
extending $p$ starting with $\left(p,\emptyset\right)$ such that
$p_{\alpha+1}|_{E_{\alpha+1}}$ forks over $E_{\alpha}$. By local
character of non-forking in simple theories (see \cite[Proposition 7.2.5]{TentZiegler}),
it follows that we must get stuck at some point in the construction,
say $\alpha$, and let $E=E_{\alpha}$, $p=p_{\alpha}$. 

Let $i_{0}<\kappa$ be a successor ordinal such that $A_{i_{0}}\neq\emptyset$,
$E\subseteq A_{i_{0}}$ and $\left|C\right|<\lambda_{i_{0}}$. (Here
we use the assumption that $\cof\left(\mu\right)=\kappa>\left|T\right|$.) 

Now we make things easier:
\begin{enumerate}
\item Enlarge $C$, so that for all $i_{0}\leq i\in\Succ\left(\kappa\right)$,
$C\cap A_{i}$ is a model of $T$. We can do this by building $C_{i,l}$
for $l<\omega$, $i_{0}\leq i<\kappa$, so that $\sequence{C_{i,l}}{i_{0}\leq i<\kappa}$
is increasing continuous, $C\subseteq C_{i,l}\subseteq C_{i,l'}$
for $l'>l$ and where $C_{i,l}\cap A_{i}$ is a model for $i\in\Succ\left(\kappa\right)$
and $\left|C_{i,l}\right|\leq\lambda_{i_{0}}$. Finally, let $C'=\bigcup\set{C_{i,l}}{i<\kappa,l<\omega}$.
\item Enlarge $C$ again, so that for all $i_{0}\leq i\in\Succ\left(\kappa\right)$,
$C\ind_{C\cap A_{i}}A_{i}$. To achieve this, build again $C_{i,l}$
as above such that for $i\in\Succ\left(\kappa\right)$, $C_{i,l}\ind_{C_{i,l}\cap A_{i}}A_{i}$
(we get this as follows. Start with $C_{i-1,l}$ and by local character
find some $B_{0}\subseteq A_{i}$ of cardinality $\leq\lambda_{i_{0}}$
such that $C_{i-1,l}\ind_{B_{0}}A_{i}$, then let $C_{i-1,l}^{1}=C_{i-1,l}\cup B_{0}$.
Continue this $\omega$ steps and take the union). Finally, let $C'=\bigcup\set{C_{i,l}}{i<\kappa,l<\omega}.$
\item Enlarge $C$ by alternating steps (1) and (2) $\omega$ times, so
that both $C\cap A_{i}$ is a model of $T$ and $C\ind_{C\cap A_{i}}A_{i}$
for all $i\in\Succ\left(\kappa\right)$. 
\end{enumerate}
Now we want to find some $e\models p$ such that $\tp\left(e/A_{i_{0}}\right)$
is a $D_{i_{0}}$-type. Start with any $e\models p$. 

Let $\alpha<\lambda_{i_{0}}$ be such that $I_{i_{0}}^{\geq\alpha}$
is indiscernible over $\left(C\cap A_{i_{0}}\right)I_{i_{0}}^{<\alpha}$.
By Ramsey, there is some $J$ with the same order type and $EM$-type
as $I_{i_{0}}^{\geq\alpha}$ over $\left(C\cap A_{i_{0}}\right)I_{i_{0}}^{<\alpha}e$
which is indiscernible over $\left(C\cap A_{i_{0}}\right)I_{i_{0}}^{<\alpha}e$.
Since $e\ind_{\left(C\cap A_{i_{0}}\right)}I_{i_{0}}$ by the choice
of $E$ and $i_{0}$ above, $e\ind_{\left(C\cap A_{i_{0}}\right)}I_{i_{0}}^{<\alpha}J$
(here we use the fact that $I_{i_{0}}^{\geq\alpha}$ is indiscernible
over $\left(C\cap A_{i_{0}}\right)I_{i_{0}}^{<\alpha}$, see also
below) . By applying an automorphism over $\left(C\cap A_{i_{0}}\right)I_{i_{0}}^{<\alpha}$
taking $J$ to $I_{i_{0}}^{\geq\alpha}$, we can find some $d\equiv_{I_{i_{0}}^{<\alpha}\left(A_{i_{0}}\cap C\right)}e$
such that $d\ind_{C\cap A_{i_{0}}}I_{i_{0}}$ and $I_{i_{0}}^{\geq\alpha}$
is indiscernible over $dI_{i_{0}}^{<\alpha}\left(C\cap A_{i_{0}}\right)$.
By the independence theorem over models in simple theories (see \cite[Theorem 7.3.11]{TentZiegler}),
as $e\ind_{A_{i_{0}}\cap C}C$, $d\ind_{A_{i_{0}}\cap C}I_{i_{0}}$,
$e\equiv_{A_{i_{0}}\cap C}d$ and $C\ind_{A_{i_{0}}\cap C}I_{i_{0}}$,
there is some $d'$ such that $d'\models p$ and $d'\equiv_{\left(C\cap A_{i_{0}}\right)I_{i_{0}}}d$. 

This gives us some $e\models p$ such that $\tp\left(e/I_{i_{0}}\left(C\cap A_{i_{0}}\right)\right)$
is a $D_{i_{0}}$-type. 

Now we use basically the same idea as in the proof of Proposition
\ref{prop:Model}, using the independence theorem: we start with a
$D_{i_{0}}$-type $q_{1}\in S\left(BI_{i_{0}}\right)$ where $C\cap A_{i_{0}}\subseteq B$
, $\left|B\right|<\lambda_{i_{0}}$, consistent with $p$, and we
want to extend it to a $D_{i_{0}}$-type $q_{2}\in S\left(BI_{i_{0}}f\right)$
where $f\in A_{i_{0}}$ which is also consistent with $p$. Let $d\models q_{1}\cup p$.
Let $\alpha<\lambda_{i}$ be such that $I_{i_{0}}^{\geq\alpha}$ is
indiscernible over $dBI_{i_{0}}^{<\alpha}$. Let $\beta>\alpha$ be
such that $I_{i_{0}}^{\geq\beta}$ is indiscernible over $fBI_{i_{0}}^{<\beta}$.
Find $J$ with the same $EM$-type as $I_{i_{0}}^{\geq\beta}$ over
$I_{i_{0}}^{<\beta}Bdf$ such that $J$ is indiscernible over $I_{i_{0}}^{<\beta}Bdf$.
Then $J\equiv_{I_{i_{0}}^{<\beta}Bd}I_{i_{0}}^{\geq\beta}$, and $J\equiv_{I_{i_{0}}^{<\beta}Bf}I_{i_{0}}^{\geq\beta}$.
As $d\models p$, we know that $d\ind_{A_{i_{0}}\cap C}I_{i_{0}}Bf$
(by choice of $d$, $E$ and $i_{0}$), hence also $d\ind_{A_{i_{0}}\cap C}I_{i_{0}}^{<\beta}JBf$
(if $\psi\left(x,j,m\right)$ witnessed forking, where $j\in J$ is
an increasing tuple and $m\in I_{i_{0}}^{<\beta}Bf$, then for some
increasing tuple $j'\in I_{i_{0}}^{\geq\beta}$, $\psi\left(d,j',m\right)$
holds. But $j'm\equiv_{A_{i_{0}}\cap C}jm$ as $I_{i_{0}}^{\geq\beta}$
is indiscernible over $fBI_{i_{0}}^{<\beta}$) . Move $J$ to $I_{i_{0}}^{\geq\beta}$
over $I_{i_{0}}^{<\beta}Bf$, to get some $d'\equiv_{I_{i_{0}}B}d$,
but now $I_{i_{0}}^{\geq\beta}$ is indiscernible over $I_{i_{0}}^{<\beta}Bd'f$
and $d'\ind_{A_{i_{0}}\cap C}I_{i_{0}}Bf.$ Now we can use the independence
theorem as above, and find $q_{2}$. 

Using this technique (constructing an increasing continuous sequence
of types over small subsets of $A_{i_{0}}$ augmented with $I_{i_{0}}$)
we can find some $e\models p$ such that $\tp\left(e/A_{i_{0}}\right)$
is a $D_{i_{0}}$-type. 

Now we may continue. More formally, we find an increasing continuous
sequence of types $p_{i}$ for $i_{0}\leq i<\kappa$ such that:
\begin{itemize}
\item $p_{i_{0}}=\tp\left(e/A_{i_{0}}\right)$; $p_{i}\in S\left(A_{i}\right)$
and for $i\in\Succ\left(\kappa\right)$, $p_{i}$ is a $D_{i}$-type
and $p_{i}\cup p$ is consistent for all $i$.
\end{itemize}
We can do this by using the same technique as in the construction
of $p_{i_{0}}$. Finally, let $p_{\kappa}=\bigcup_{i<\kappa}p_{i}$,
let $e\models p_{\kappa}$, and let $B_{i}=\emptyset$ for $i<i_{0}$
and $A_{i}e$ for $i\geq i_{0}$. 
\end{proof}

\begin{proof}
[Proof of Theorem \ref{thm:Main-Simple theories}.]

Let $\sequence{C_{\alpha}}{\alpha<\mu^{+}}$ be a sequence as in Remark
\ref{rem:Square'}. Note that $\left|C_{\alpha}\right|<\mu$ for all
$\alpha<\mu^{+}$ as $\mu$ is singular. Let $\set{S_{\alpha}}{\alpha<\mu^{+}}$
be a partition of $\mu^{+}$ to sets of size $\mu^{+}$. We construct
a sequence $\sequence{\left(\bar{A}_{\alpha},\bar{p}_{\alpha}\right)}{\alpha<\mu^{+}}$
such that:
\begin{enumerate}
\item \label{enu:in M}$\bar{A}_{\alpha}=\sequence{A_{\alpha,i}}{i<\kappa}\in\SquareClass$;
\item $\bar{p}_{\alpha}$ is an enumeration $\sequence{p_{\alpha,\beta}}{\beta\in S_{\alpha}\backslash\alpha}$
of all complete types over subsets of $\bigcup_{i}A_{\alpha,i}$ of
size $<\mu$ (this uses $\mu^{+}=2^{\mu}$); 
\item \label{enu:star}If $\beta<\alpha$ then $\bar{A}_{\beta}\leqstar\bar{A}_{\alpha}$
(see Definition \ref{def:order}); 
\item \label{enu:realizing types}If $\alpha\in S_{\gamma}$ and $\gamma\leq\alpha$,
then $\bar{A}_{\alpha+1}$ contains a realization of $p_{\gamma,\alpha}$;
\item \label{enu:limitsquare}If $\alpha$ is a limit ordinal, then for
all $i<\kappa$ such that $\left|C_{\alpha}\right|<\lambda_{i}$,
$A_{\alpha,i}\neq\emptyset$ and for all $\beta\in C_{\alpha}$, $\bar{A}_{\beta}\leq_{i}\bar{A}_{\alpha}$. 
\end{enumerate}
Start with $A_{0,i}=I_{i}$ for $i\in\Succ\left(\kappa\right)$ and
otherwise defined by continuity. 

For $\alpha+1$, use Main Lemma \ref{lem:Main Lemma simple}. 

For $\alpha$ limit there are two possibilities.
\begin{casenv}
\item $\sup\left(C'_{\alpha}\right)=\alpha$. Suppose $i_{0}<\kappa$ is
minimal such that $\left|C_{\alpha}\right|<\lambda_{i_{0}}$ (so necessarily
$i_{0}\in\Succ\left(\kappa\right)$). For $i<i_{0}$, let $A_{\alpha,i}=\emptyset$.
For $i\geq i_{0}$ successor, let $A_{\alpha,i}=\bigcup_{\beta\in C_{\alpha}}A_{\beta,i}$.
Note that $\left|A_{\alpha,i}\right|\leq\lambda_{i}$. We have to
show that $\bar{A}_{\alpha}$ satisfies (\ref{enu:in M}), (\ref{enu:star})
and (\ref{enu:limitsquare}). The latter is by construction and the
fact that for $\beta\in C_{\alpha}$, $\left|C_{\beta}\right|\leq\left|C_{\alpha}\right|$.

For (\ref{enu:in M}), suppose $s\subseteq A_{\alpha,i}$ is a finite
set where $i_{0}\leq i\in\Succ\left(\kappa\right)$. For every element
$e\in s$, there is some $\beta_{e}\in C_{\alpha}$ such that $e\in A_{\beta_{e},i}$.
Let $\beta=\max\set{\beta_{e}}{e\in s}$. Then $\beta$ is a limit
ordinal and $C_{\alpha}\cap\beta=C_{\beta}$. As $\left|C_{\beta}\right|<\lambda_{i_{0}}$,
it follows by the induction hypothesis that $s\subseteq A_{\beta,i}$.
As $A_{\beta,i}$ is a $D_{i}$-set for all such $\beta$, it follows
that $A_{\alpha,i}$ is a $D_{i}$-set as well.

Lastly, (\ref{enu:star}) is easy by assumption of the case and transitivity
of $\leqstar$.

\item $\sup\left(C'_{\alpha}\right)<\alpha$. In this case, if $C_{\alpha}\neq\emptyset$,
then it has a last element, and $\cof\left(\alpha\right)=\omega<\kappa$.
If $C_{\alpha}=\emptyset$, choose $\gamma=0$, otherwise, it is the
last element of $C_{\alpha}$. Let $i^{*}<\kappa$ be minimal such
that $\left|C_{\alpha}\right|<\lambda_{i^{*}}$.

Choose a cofinal set $S\subseteq\alpha$ above $\gamma$ of size $\aleph_{0}<\kappa$.
For all $\varepsilon<\zeta\in S$, $\bar{A}_{\varepsilon}\leqstar\bar{A}_{\zeta}$
as is witnessed by some $i_{\varepsilon,\zeta}<\kappa$. As $\kappa$
is regular, there is some $i^{*}<i_{0}\in\Succ\left(\kappa\right)$
such that $\bar{A}_{\varepsilon}\leq_{i_{0}}\bar{A_{\zeta}}$ for
all $\varepsilon<\zeta\in S$. By the same reasoning there is some
$i_{0}<i_{1}\in\Succ\left(\kappa\right)$ such that $\bar{A}_{\gamma}\leq_{i_{1}}\bar{A}_{\varepsilon}$
for all $\varepsilon\in S$. Set $A_{\alpha,i}=A_{\gamma,i}$ for
$i<i_{1}$ and $A_{\alpha,i}=\bigcup_{\beta\in S}A_{\beta,i}$ for
$i\geq i_{1}$.

Now: (\ref{enu:in M}) follows by choice of $i_{0}$ (so that each
$A_{\alpha,i}$ is a $D_{i}$-set) and $i_{1}$ (so that $A_{\alpha,i}$
is increasing with $i$), (\ref{enu:star}) follows by the transitivity
of $\leqstar$, so we are left with (\ref{enu:limitsquare}). The
first part is easy: if $C_{\alpha}=\emptyset$, then it follows by
our choice of $\bar{A}_{0}$. Otherwise, use the fact that $C_{\gamma}\subseteq C_{\alpha}$.

Suppose $\beta\in C_{\alpha}$. Then, either $\beta=\gamma$, in which
case this clause is obvious, or $\beta<\gamma$, in which case $\beta\in C_{\gamma}$.
By the induction hypothesis, $\bar{A}_{\beta}\leq_{i^{*}}\bar{A}_{\gamma}\leq\bar{A}_{\alpha}$,
so we are done by the choice of $i^{*}$. 

\end{casenv}
Finally, let $M=\bigcup_{\alpha<\mu^{+},i<\kappa}A_{\alpha,i}$. Then
$M$ is a $\mu$-saturated model of $T$ by (\ref{enu:realizing types}).
However, it is not $\mu^{+}$-saturated because the type $\set{\varphi\left(x,a_{j}\right)}{j\in I\mbox{ even}}\cup\set{\neg\varphi\left(x,a_{j}\right)}{j\in I\mbox{ odd}}$
is not realized in $M$: suppose $b$ realizes it. $b$ is a finite
tuple, but since $\bar{A}_{\alpha}$ is an increasing continuous sequence
for all $\alpha<\mu^{+}$, there must be some $\alpha<\mu^{+}$ and
$i\in\Succ\left(\kappa\right)$ such that $b\in A_{\alpha,i}$. But
clearly $\tp\left(b/I_{i}\right)$ is not a $D_{i}$-type --- contradiction. 
\end{proof}

\section{Dependent theories}

Here we characterize the NIP (dependent) theories which have exact
saturation at $\kappa$, assuming the Continuum Hypothesis at $\kappa$.
They happen to be precisely the non-distal theories. 

Throughout this section, assume that $T$ is NIP: for no formula $\varphi\left(x,y\right)$
is it the case that there are $\sequence{a_{i}}{i<\omega}$ and $\sequence{b_{s}}{s\subseteq\omega}$
such that $\C\models\varphi\left(a_{i},b_{s}\right)$ holds iff $i\in s$. 

We use the notation $X^{\opp}$ for a linear order $X$ to denote
$X$ with the order reversed.

\subsection{Preliminaries}

\subsubsection{NIP theories.}

Suppose $I$ is an indiscernible sequence. We will identify $I$ and
its underlying order. For instance, we will say that $I$ is \emph{dense}
if its underlying order type is. All our sequences will be infinite.

\subsubsection*{Shrinking of indiscernibles}

Recall the following definition, which, in NIP theories, gives a complete
type. 
\begin{defn}
\label{def:Average} The \emph{average type (at $\infty$)} of an
indiscernible sequence with no end, $\sequence{a_{i}}{i\in I}$, over
$A$, denoted by $\average IA{\infty}$, consists of formulas of the
form $\phi\left(b,x\right)$ with $b\in A$, such that for some $i\in I$,
$\C\models\phi\left(b,a_{j}\right)$ for every $j\ge i$. 
\end{defn}
We will use shrinking of indiscernibles (which is a stronger version
of the existence of averages), as formulated in the following fact.
Given a linear order $\left(X,<\right)$, a \emph{finite convex equivalence
relation }on $X$ is an equivalence relation with finitely many classes
which are convex. 
\begin{fact}
\label{fact:(Shrinking)}(Shrinking, see e.g., \cite[Theorem 3.33]{pierrebook})
Suppose $I$ is an indiscernible sequence over some set $A$, and
suppose that $b$ is some finite tuple and $\Delta$ a finite set
of formulas. Then there is a finite convex equivalence relation $\twiddle$
on $I$ such that each $\twiddle$-class $C$ is $\Delta$-indiscernible
over $Ab\cup I\backslash C$. 

Moreover, given a formula $\varphi\left(x_{0},\ldots,x_{n-1},y,z\right)$
there is such an equivalence relation $\twiddle$ such that for any
two finite increasing sequences $\bar{i},\bar{j}$ of length $n$
from $I$, if $\bar{i}\sim\bar{j}$ (i.e., $i_{0}\sim j_{0},\ldots,i_{n-1}\sim j_{n-1}$)
and $a\in A^{z}$ then $\varphi\left(a_{i_{0}},\ldots,a_{i_{n-1}},b,a\right)$
holds iff $\varphi\left(a_{j_{0}},\ldots,a_{j_{n-1}},b,a\right)$. 
\end{fact}
A cut in an indiscernible sequence $I$ has the form $\cut c=\left(I_{1},I_{2}\right)$
for $I_{1}$ an initial segment of $I$ and $I_{2}$ its corresponding
end segment. Our cuts will always be internal (both $I_{1},I_{2}$
are not empty) and have infinite cofinality from both sides, unless
we specifically say otherwise.

Here is a useful and easy corollary of Shrinking of indiscernibles
(Fact \ref{fact:(Shrinking)}). We leave its proof as an exercise. 
\begin{cor}
\label{cor:cor to shrinking}Suppose that $I$ is indiscernible, $\left|T\right|\leq\theta$,
and that $\cut c_{i}$ for $i<\theta^{+}$ are distinct cuts, each
of cofinality at least $\theta^{+}$ from both sides. Then for any
set $A$ with $\left|A\right|\leq\theta$, there is some $i<\theta^{+}$
such that there is an interval $I_{0}$ around $\cut c_{i}$ which
is indiscernible over $A\cup I\backslash I_{0}$. 
\end{cor}

\subsubsection*{Invariant types and Morley sequences}

Recall that for a global $A$-invariant type $p\in S\left(\C\right)$,
and for $B\supseteq A$, the sequence $\sequence{a_{i}}{i<\omega}$
generated by realizing $a_{i}\models p|_{Ba_{<i}}$ is always indiscernible
over $B$. This sequence is a\emph{ Morley sequence generated by $p$
over $B$}. In general, Morley sequences need not be of order type
$\omega$. A sequence $\sequence{a_{i}}{i\in I}$ of any order type
$\left(I,<\right)$ is a \emph{Morley sequence of $p$ over $B$}
if for any $i\in I$, $a_{i}\models p|_{Ba_{<i}}$. Let $p^{\left(I\right)}|_{B}=\tp\left(\sequence{a_{i}}{i\in I}/B\right)$
and $p^{\left(I\right)}$ be the global $A$-invariant type $\bigcup_{A\subseteq B}p^{\left(I\right)}|_{B}$.
It is an exercise to show that this is well defined. 

In NIP, any indiscernible sequence over $A$ is a Morley sequence
of some invariant type over a set containing $A$ (extend $I$ to
$I+I^{\opp}$, and let $p$ be the average type of $I^{\opp}$ at
$-\infty$, so $p$ is $I^{\opp}$-invariant and $I$ is a Morley
sequence of $p$ over $AI^{\opp}$).

\subsubsection{Distal theories.}

Suppose that $I$ is indiscernible and that $\cut c=\left(I_{1},I_{2}\right)$
is a cut in $I$. For a set $A$, denote by $\lim\left(\cut c^{-}/A\right)=\average{I_{1}}A{\infty}$,
and similarly $\lim\left(\cut c^{+}/A\right)$ is the average type
of $I_{2}$ at $-\infty$ over $A$. This is the \emph{limit type
of $\cut c^{-}$ (or $\cut c^{+}$) over $A$}. Note that if $A=\C$,
this is an $I_{1}$ (or $I_{2}$)-invariant type. 

Note that $\lim\left(\cut c^{+}/I\right)=\lim\left(\cut c^{-}/I\right)$
(so we just write $\lim\left(\cut c/I\right)$), and that if $b\models\lim\left(\cut c/I\right)$
then $b$ \emph{fills} $I$: when $b$ is put in $\cut c$, the augmented
sequence $I\cup b$ is indiscernible. In fact this is equivalent to
satisfying $\lim\left(\cut c/I\right)$. 

More generally, if $\bar{c}$ is some (possibly infinite) ordered
tuple of tuple in same length as the tuples in $I$, we will say that
$\bar{c}$ \emph{fills} $\cut c$ if when we put $\bar{c}$ in $\cut c$,
in the right order, the augmented sequence $I\cup\bar{c}$ is indiscernible. 

For instance, if $\bar{c}$ is of order type $\omega$, then, using
the notation from above, we have that if $\bar{c}\models\lim\left(\cut c^{+}/\C\right)^{\left(\omega\right)}|_{I}$,
then $\bar{c}$ fills $\cut c$. 

When $I$ is indiscernible over $A$, we can add ``over $A$'' everywhere,
meaning that we name the elements of $A$. 
\begin{defn}
\label{def:orthogonal}We say that two types $p(x),q(y)\in S\left(A\right)$
are \emph{orthogonal} if their union implies a complete type in $x,y$
over $A$ (usually this notion is called ``weakly orthogonal'',
but full orthogonality will not be used in this paper). 
\end{defn}
If $\cut c_{1}$ and $\cut c_{2}$ are two distinct cuts in a dense
indiscernible sequence $I$, and $b_{i}\models\lim\left(\cut c_{i}/I\right)$
for $i=1,2$, we will say that $b_{1},b_{2}$ are \emph{$I$-independent},
if, when placed in their appropriate cuts, $I\cup\left\{ b_{1},b_{2}\right\} $
is indiscernible. This is equivalent to saying that the two limit
types are orthogonal. 
\begin{defn}
A dense indiscernible sequence $I$ is called \emph{distal} if whenever
$\cut c_{1},\cut c_{2}$ are two distinct cuts in $I$, then $\lim\left(\cut c_{1}/I\right)$
and $\lim\left(\cut c_{2}/I\right)$ are orthogonal. \end{defn}
\begin{rem}
\label{rem:about distal indiscernible sequences}
\begin{enumerate}
\item If $I$ is a dense indiscernible sequence which is not distal, then
there are two distinct cuts $\cut c_{i}$ for $i=1,2$ and $b_{i}\models\lim\left(\cut c_{i}/I\right)$
such that $I\cup\left\{ b_{1},b_{2}\right\} $ (i.e., $I$ augmented
with $b_{1},b_{2}$ placed in their corresponding cuts) is not indiscernible.
By compactness and indiscernibility, it is easy to see that this is
true for any distinct cuts $\cut d_{1},\cut d_{2}$. 
\item Distal indiscernible sequences and distal theories (see below) were
defined and discussed at length in \cite{Distal}. There, they are
defined a bit differently, namely: an infinite indiscernible sequence
(not necessarily dense) is distal if it has the same EM-type as a
dense indiscernible sequence which is distal. On the face of it, this
defines a larger class even inside dense sequences, but this is not
the case by \cite[Lemma 2.3]{Distal}. 
\end{enumerate}
\end{rem}

\begin{defn}
An NIP theory $T$ is called \emph{distal} if all infinite indiscernible
sequences in it are distal. \end{defn}
\begin{fact}
\label{fact:distal - isolation}\cite[Theorem 9.22, Theorem 21]{pierrebook,PairsII}A
theory $T$ is distal iff for any formula $\varphi\left(x,y\right)$
there is a formula $\theta\left(x,z\right)$ such that: for any $M\models T$,
$A\subseteq M$ of size at least $2$, $a\in M^{x}$ and a finite
$C\subseteq A^{y}$ there is $b\in A^{z}$ such that $M\models\theta\left(a,b\right)$
and $\theta\left(x,b\right)\vdash\tp_{\varphi}\left(a/C\right)$.
Equivalently, in some  elementary extension $\left(M,A\right)\prec\left(M',A'\right)$
of the pair, there is $b\in\left(A'\right)^{z}$ such that $M'\models\theta\left(a,b\right)$
and $\theta\left(x,b\right)\vdash\tp_{\varphi}\left(a/A\right)$. \end{fact}
\begin{example}
Examples of distal theories include o-minimal theories (e.g., RCF,
DLO), and the theory of the $p$-adics. 
\end{example}

\subsection{Results}

Now we are ready to state our main theorem for this section.
\begin{thm}
\label{thm:Distal-exact saturation}Suppose that $\kappa$ is a singular
cardinal such that $\kappa^{+}=2^{\kappa}$. An NIP theory $T$ with
$\left|T\right|<\kappa$ is distal iff it does not have exact saturation
at $\kappa$. 
\end{thm}
In \cite[Question 2.30]{Sh900}, the following question appears. Is
there a dependent theory $T$ with exact saturation at some singular
cardinal $\kappa$ of cofinality $>\left|T\right|$ such that even
in $T^{\eq}$ there is no infinite indiscernible set?

In \cite[Section 9.3.4]{pierrebook}, there is an example of an NIP
theory which is not distal and yet has no non-trivial generically
stable type, even in $T^{\eq}$. Having an infinite indiscernible
set is equivalent to having a generically stable type (see \cite[Section 2.2.2, Remark 2.32]{pierrebook}).
Together we get the following.
\begin{cor}
\label{cor:Answer to 900}The answer to \cite[Question 2.30]{Sh900}
is ``yes'', provided that for some singular $\kappa$ with $\cof\left(\kappa\right)>\left|T\right|$,
$\kappa^{+}=2^{\kappa}$. 
\end{cor}

\subsubsection*{Left to Right}
\begin{prop}
If $T$ is distal and $\left|T\right|<\kappa$ is singular, then every
$\kappa$-saturated model $M$ is also $\kappa^{+}$-saturated.\end{prop}
\begin{proof}
Suppose $A\subseteq M$ and $\left|A\right|=\kappa$. Suppose $\mu=\cof\left(\kappa\right)<\kappa$.
Let $p\in S\left(A\right)$, and we want to show that $p$ is realized
in $M$. Write $A=\bigcup_{i<\mu}A_{i}$. For each $i<\mu$, let $b_{i}\models p\upharpoonright A_{i}$,
$b_{i}\in M$ (exists as $M$ is $\kappa$-saturated). Let $\varphi\left(x,y\right)$
be some formula. By Fact \ref{fact:distal - isolation}, there is
some $\left(M_{i}',A_{i}'\right)\succ\left(M,A_{i}\right)$, $d_{i}^{\varphi}\in A'_{i}$
and $\theta^{\varphi}$ such that $M'_{i}\models\theta^{\varphi}\left(b_{i},d_{i}^{\varphi}\right)$
and $\theta^{\varphi}\left(x,d_{i}^{\varphi}\right)\vdash\tp_{\varphi}\left(b_{i}/A_{i}\right)$
(as usual, we assume that everything happens in the monster model
$\C$ of $T$). Let $q_{i}=\set{\theta^{\varphi}\left(x,d_{i}^{\varphi}\right)}{\varphi\left(x,y\right)\in T}$. 
\begin{claim}
\label{claim:increasing types}If $\mu>j\geq i$ then $b_{j}\models q_{i}$
(in $M_{i}'$).\end{claim}
\begin{proof}
For $j=i$, this is by choice of $\theta^{\varphi}$, so suppose $j>i$
and that $\theta^{\varphi}\left(b_{i},d_{i}^{\varphi}\right)\land\neg\theta^{\varphi}\left(b_{j},d_{i}^{\varphi}\right)$
for some $\varphi$. Hence $\left(M_{i}',A_{i}'\right)\models\exists z\in P\left(\theta^{\varphi}\left(b_{i},z\right)\land\neg\theta^{\varphi}\left(b_{j},z\right)\right)$
where $P$ is a predicate symbol interpreted as $A_{i}'$. Hence the
same is true in $\left(M,A_{i}\right)$. But $b_{i}\equiv_{A_{i}}b_{j}$
so this cannot happen. 
\end{proof}
Let $d_{i}=\sequence{d_{i}^{\varphi}}{\varphi\in T}$ for $i<\mu$,
and find $e_{i}\models\tp\left(d_{i}/A_{i}\cup\set{b_{i}}{i<\mu}\right)$
in $M$, which exists by $\kappa$-saturation. Enumerate it as $e=\set{e_{i}^{\varphi}}{\varphi\in T}$.
Let $r_{i}\left(x\right)=\set{\theta^{\varphi}\left(x,e_{i}^{\varphi}\right)}{\varphi\in T}$.
Note that for each $\varphi$ and $i<\mu$, $\theta^{\varphi}\left(x,e_{i}^{\varphi}\right)\vdash\tp_{\varphi}\left(b_{i}/A_{i}\right)$
and that $b_{j}\models r_{i}$ for $j\geq i$. 

Let $r=\bigcup_{i<\mu}r_{i}$. By Claim \ref{claim:increasing types},
$r$ is a consistent type in $M$, and it is a type over a set of
size $\leq\mu\cdot\left|T\right|<\kappa$, so it is realized, say
by $c\in M$. Then $c\models p$. 
\end{proof}

\subsubsection*{Right to left --- Technical lemmas}
\begin{defn}
\label{def:generic cuts}Suppose $s\subseteq\C$ is a finite set,
and $I$ is an indiscernible sequence. Let us say that a cut $\cut c$
in $I$ is\emph{ \generic{} for $s$} if there is a neighborhood
$I_{0}=\left(i_{1},i_{2}\right)$ of $\cut c$ in $I$ such that $I_{0}$
is indiscernible over $s\cup I\backslash I_{0}$. 

Similarly, for a small set $A$, $\cut c$ is $A$-\generic{} if
it is \generic{} for every finite subset from $A$. We can similarly
say the $\cut c$ is $q$-\generic{} for a complete finitary type
$q$ over $I$, meaning that $\cut c$ is $c$-\generic{} for some
(any) $c\models q$ (formally, for $\bigcup c$). 
\end{defn}

\begin{defn}
\label{def:NIP finite diagram}For an indiscernible sequence $I$,
let $D_{I}$ be the collection of finitary types $q\in S^{<\omega}\left(I\right)$
such that $q$ is orthogonal (see Definition \ref{def:orthogonal})
to $\lim\left(\cut c/I\right)$ for some $q$-\generic{} cut $\cut c$. \end{defn}
\begin{rem}
Suppose that $\cut c$ is \generic{} for $c$ in some sequence $I$
as witnessed by $\left(i_{1},i_{2}\right)$. Then\emph{ $\lim\left(\cut{c^{-}}/Ic\right)=\lim\left(\cut c^{+}/Ic\right)$}
and if $b\models\lim\left(\cut c/Ic\right)$ then $b$ fills $\cut c$
and moreover, the interval $\left(i_{1},i_{2}\right)\cup b$ is indiscernible
over $c\cup I\backslash\left(i_{1},i_{2}\right)$. 

Hence, to say that $q$ is orthogonal to $\cut c$ in Definition \ref{def:NIP finite diagram}
means that whenever $c\models q$ and $a\models\lim\left(\cut c/I\right)$,
we have that $a\models\lim\left(\cut c/Ic\right)$. 
\end{rem}

\begin{rem}
\label{rem:reverse}Note that if $I$ is dense indiscernible, and
$\cut c$ is $c$-\generic{} then $\cut c$ is $c$-\generic{} also
in $I^{\opp}$ ($I$ with reverse order). Similarly, $q\in D_{I}$
iff $q\in D_{I^{\opp}}$ with the same cut witnessing this. Moreover,
note that $I$ is a $D_{I}$-set (for any finite $s\subseteq I$,
any cut $\cut c$ will witness this), and that if $A$ is a $D_{I}$-set
then so is $AI$. 
\end{rem}
We want to show that this definition behaves well.
\begin{mainlemma}
\label{lem:Main Lemma}Suppose $I$ is dense indiscernible. If $q\in D_{I}$
and $\cut c$ is $q$-\generic{} then $q$ is orthogonal to $\lim\left(\cut c/I\right)$. 
\end{mainlemma}
The proof uses two ingredients, both from \cite{Distal}. One is the
finite co-finite theorem, and the other is the external characterization
of domination. 

First, a technical claim.

\begin{claim}
\label{claim:preserving D_I}Suppose that $I$ is a dense indiscernible
sequence, and that $\tp\left(c/I\right)\in D_{I}$ as witnessed by
$\cut c=\left(I_{1},I_{2}\right)$. Let $I_{1}\ni i_{1}<i_{2}\in I_{2}$
be such that $\left(i_{1},i_{2}\right)$ is indiscernible over $I\backslash\left(i_{1},i_{2}\right)\cup c$.
Suppose that $J$ is some dense sequence with no minimum such that
$I'=I_{1}+J+I_{2}$ is indiscernible and, in $I'$ we still have that
$\left(i_{1},i_{2}\right)$ is indiscernible over $I\backslash\left(i_{1},i_{2}\right)\cup c$.
Then $\tp\left(c/I'\right)\in D_{I'}$ as witnessed by $\left(I_{1},J+I_{2}\right)$. 

Similarly, if $J$ has no maximum, the same is true for $\left(I_{1}+J,I_{2}\right)$. \end{claim}
\begin{proof}
What this claim says is that, letting $\cut d=\left(I_{1},J+I_{2}\right)$,
$\cut d$ is $c$-\generic{} in $I'$ (by assumption) and if $a\models\lim\left(\cut d/I'\right)$
then $a\models\lim\left(\cut d/I'c\right)$. Suppose not. Then for
some formula $\varphi\left(x,x_{1},\ldots,x_{k-1}\right)$ over $c\cup I_{1}\cup I\backslash\left(i_{1},i_{2}\right)$
there are $\cut d<b_{1}<\ldots<b_{k-1}<i_{2}$ from $J+I_{2}$ such
that $\neg\varphi\left(a,b_{1},\ldots,b_{k-1}\right)$ holds even
though for all $c_{0}<\ldots<c_{k-1}<i_{2}$ from $I_{2}$, $\varphi\left(c_{0},\ldots,c_{k-1}\right)$
holds.

Choose $b_{1}'<\ldots<b_{k-1}'<i_{2}$ be from $I_{2}$. Let $\Gamma_{0}\left(x\right)$
be the type over $c\cup I$ saying that $x$ fills $\cut c$ and let
$\Gamma=\Gamma_{0}\cup\left\{ \neg\varphi\left(x,b_{1}',\ldots,b_{k-1}'\right)\right\} $.
Then $\Gamma$ is consistent: a finite part $\Gamma_{0}'$ of $\Gamma_{0}$
says that $\sequence{a'_{j}}{j<m}\concat\left\langle x\right\rangle \concat\sequence{c_{j}'}{j<l}$
is $\Delta$-indiscernible where $\sequence{a'_{j}}{j<m}$, $\sequence{c'_{j}}{j<l}$
are increasing sequences from $I_{1}$ and $I_{2}$ respectively and
$\Delta$ is a finite set of formulas. For $0<i<k$, let $s'_{i}=\set{j<l}{b_{i-1}'<c_{j}'\leq b'_{i}}$
where $b_{0}'=-\infty$ in $I_{2}$. As $J+I_{2}$ is dense, we can
find an increasing sequence $\sequence{c_{j}}{j<l}$ from $J+I_{2}$
such that, letting $s_{i}=\set{j<l}{b_{i-1}<c_{j}<b_{i}}$ (where
$b_{0}=-\infty$ in $J$), $s_{i}'=s_{i}$. As $\left(i_{1},i_{2}\right)$
is indiscernible over $c\cup I\backslash\left(i_{1},i_{2}\right)$
in $I'$, there is an automorphism $\sigma$ fixing $c\cup I_{1}\cup I\backslash\left(i_{1},i_{2}\right)$
taking $b_{i}'$ to $b_{i}$ and $c'_{j}$ to $c_{j}$. But $\sigma^{-1}\left(a\right)$
then satisfies $\Gamma_{0}'$ and also $\neg\varphi\left(x,b_{1}',\ldots,b_{k-1}'\right)$
as we wanted. 

However, $\Gamma$ cannot be satisfied by the assumption that $\tp\left(c/I\right)\in D_{I}$. 

The analogous claim on $\left(I_{1}+J,I_{2}\right)$ is proved similarly. 
\end{proof}
We continue with the finite-co-finite theorem \cite[Theorem 3.30]{Distal}.
This theorem states that if $I=I_{1}+I_{2}+I_{3}$ is indiscernible,
both $I_{1},I_{3}$ are infinite and $I_{1}+I_{3}$ is indiscernible
over $A$, then for any $a\in A$, and any $\varphi\left(x,y\right)$,
the set $\set{b\in I_{2}}{\C\models\varphi\left(a,b\right)}$ is either
finite or co-finite. 
\begin{fact}
\label{fact:corollary of finite-co-finite}\cite[Corollary 3.32]{Distal}
Let $I_{1}+I_{2}+I_{3}$ be an indiscernible sequence, such that $I_{1}$
and $I_{3}$ are without endpoints. Suppose that $I_{1}+I_{3}$ is
indiscernible over $A$. Then there is some $I_{2}'\subseteq I_{2}$
of size $\leq\left|T\right|+\left|A\right|$ such that $I_{1}+I_{2}'+I_{3}$
is $A$-indiscernible. \end{fact}
\begin{prop}
\label{prop:Orthogonal to any sequence}Let $I$ be dense indiscernible.
Suppose that $q\in D_{I}$ and that $\cut c$ is a cut in $I$ that
witnesses it. Then $q$ is orthogonal to $\lim\left(\cut c^{+}/\C\right)^{\left(X\right)}|_{I}$
for any linear order $\left(X,<\right)$.\end{prop}
\begin{proof}
Suppose $\left(i_{1},i_{2}\right)$ is an interval around $\cut c$
which witnesses that $q\in D_{I}$. What we have to show is that if
$c\models q$ and $\bar{a}\models\lim\left(\cut c^{+}/\C\right)^{\left(X\right)}|_{I}$
(i.e., $\bar{a}$ is an ordered sequence of order type $\left(X,<\right)$
which fills $\cut c$), then $\left(i_{1},i_{2}\right)\cup\bar{a}$
is indiscernible over $c\cup I\backslash\left(i_{1},i_{2}\right)$. 

We may assume that $X$ is finite. We prove by induction on $n$ that
the proposition holds for all $I$ with $X=n$ as an order. For $n=1$
this is just the assumption that $q\in D_{I}$. Suppose this is true
for $n$ and prove it for $n+1$. 

Let $\cut c=\left(I_{1},I_{2}\right)$ inside $\left(i_{1},i_{2}\right)$.
Let $a_{0},\ldots,a_{n}$ fill $\cut c$. 

Let $J$ be a dense indiscernible sequence of cofinality $\left(\left|I\right|+\left|T\right|\right)^{+}$
from both sides (and such that between any two elements there are
$\left(\left|I\right|+\left|T\right|\right)^{+}$ elements), such
that $I_{1}+a_{0}+\ldots+a_{n-1}+J+a_{n}+I_{2}$ is indiscernible
over $I\backslash\left(i_{1},i_{2}\right)$. By Fact \ref{fact:corollary of finite-co-finite},
we may assume that $I_{1}+J+I_{2}$ is indiscernible over $c\cup I\backslash\left(i_{1},i_{2}\right)$.

Let $\cut d=\left(I_{1}+J,I_{2}\right)$, which we identify with the
corresponding cut in the extended sequence $I'=I\cup J$. By Claim
\ref{claim:preserving D_I}, $\cut d$ witnesses that $\tp\left(c/I'\right)\in D_{I'}$,
and hence $I_{1}+J+a_{n}+I_{2}$ is indiscernible over $c\cup I\backslash\left(i_{1},i_{2}\right)$.
Note that $J+a_{n}$ is dense with no minimum, so by applying Claim
\ref{claim:preserving D_I} again on the sequence $I''=I'\cup a_{n}$,
we get that $\tp\left(c/I''\right)\in D_{I''}$ as witnessed by the
cut corresponding to $\left(I_{1},J+a_{n}+I_{2}\right)$. By the induction
hypothesis, $I_{1}+a_{0}+\ldots+a_{n-1}+J+a_{n}+I_{2}$ is indiscernible
over $I\backslash\left(i_{1},i_{2}\right)$. In particular, we get
what we wanted.  
\end{proof}
Now we need to discuss domination in indiscernible sequences. Although
we will not use it directly, we give the definition.
\begin{defn}
\label{def:domination}Suppose that $I$ is a dense Morley sequence
of an $A$-invariant type $p$. Suppose $\cut c$ is a cut in $I$,
$\bar{a}$ is an ordered tuple which fills $\cut c$ over $A$, and
that $a\models p|_{IA}$. We will say that $\bar{a}$ dominates $a$
over $\left(I,A\right)$ if whenever $\cut d\neq\cut c$ is a cut
in $I$, and $\bar{b}$ is any ordered tuple which fills $\cut d$
over $A$, if $\bar{a}\ind_{I}\bar{b}$ (i.e., when put in their appropriate
cuts in the right order, $I\cup\left\{ \bar{a},\bar{b}\right\} $
is $A$-indiscernible) then $a\ind_{I}\bar{b}$ (i.e., $a\models p|_{IAb}$). 

Say that $\bar{a}$ strongly dominates $a$ over $\left(I,A\right)$
if for any dense $A$-indiscernible sequence $J$ containing $I$
such that $\bar{a}$ still fills some cut in $J$ and $a\models p|_{AJ}$,
$\bar{a}$ dominates $a$ over $\left(J,A\right)$. \end{defn}
\begin{rem}
In the definition in \cite[Definition 3.2]{Distal}, the tuple $\bar{b}$
is dense. However, by compactness, if $\bar{a}$ dominates $a$ over
$\left(I,A\right)$ for dense tuples, then there is domination for
any ordered tuple. 
\end{rem}
The following fact says that we can always find strong domination.
\begin{fact}
\label{fact:Existence of Domination}\cite[Proposition 3.6]{Distal}Let
$I$ be a dense Morley sequence of $p$ over $A$ and $a\models p|_{AI}$.
Suppose $\cut c$ is a cut in $I$. Then there is an ordered tuple
$\bar{a}$ that fills $\cut c$ which dominates $a$ over $\left(I,A\right)$. 
\end{fact}

\begin{fact}
\label{fact:extenal domination}\cite[Proposition 3.7]{Distal}Let
$I$ be a dense Morley sequence of an $A$-invariant type $p$ over
$A$, $a\models p|_{AI}$. Suppose $\bar{a}$ fills a cut $\cut c$
and that $\bar{a}$ strongly dominates $a$ over $\left(I,A\right)$.
Let $d\in\C$. Assume that:

There is a partition $I=J_{1}+J_{2}+J_{3}+J_{4}$ such that $J_{2}$
and $J_{4}$ are infinite, $\cut c$ is interior to $J_{2}$, $J_{2}\cup\left\{ \bar{a}\right\} $
is indiscernible over $J_{\neq2}Ad$ and $J_{4}$ is a Morley sequence
of $p$ over $J_{\neq4}Ad$. 

Then $a\models p|_{AId}$. \end{fact}
\begin{proof}
[Proof of Main Lemma \ref{lem:Main Lemma}]

So assume that $I$ is dense indiscernible and $\tp\left(c/I\right)\in D_{I}$.
Suppose that $\cut c$ is a cut which witnesses this. 

Let $\cut d$ be another cut which is \generic{} for $c$. We want
to show that $\tp\left(c/I\right)$ is orthogonal to $\lim\left(\cut d/I\right)$.
Suppose without loss of generality that $\cut c<\cut d$ (otherwise,
reverse the order of $I$, see Remark \ref{rem:reverse}). Suppose
that $a\models\lim\left(\cut d/I\right)$. We want to show that $a\models\lim\left(\cut d/Ic\right)$.
Let $i_{1}<i_{2}<j_{1}<j_{2}$ witness that $\cut c$ and $\cut d$
are $c$-\generic{} respectively. Let $\cut d=\left(I_{1},I_{2}\right)$.
Let $p=\lim\left(\cut d^{+}/\C\right)$, so that $p$ is $I_{2}$-invariant.
We want to show that $a\models p|_{Ic}$. 

By Fact \ref{fact:Existence of Domination}, we may find some $\bar{a}$
filling $\cut c$ over $I_{2}$ such that $\bar{a}$ strongly dominates
$a$ over $\left(I,I_{2}\right)$. As $\cut c$ witnesses that $\tp\left(c/I\right)\in D_{I}$,
and as $\bar{a}$ fills $\cut c$ in $I$, Proposition \ref{prop:Orthogonal to any sequence}
implies that $\left(i_{1},i_{2}\right)\cup\bar{a}$ is indiscernible
over $cI\backslash\left(i_{1},i_{2}\right)$. 

Let $J_{2}=\left(i_{1},i_{2}\right)$, $J_{4}=\left(j_{2},+\infty\right)$
in $I_{1}$ and let $J_{1},J_{3}$ fill the other parts of $I_{1}$
so that $I_{1}=J_{1}+J_{2}+J_{3}+J_{4}$. Let us check that the assumptions
of Fact \ref{fact:extenal domination} are satisfied, with $I$ there
being $I_{1}$, $d=c$ and $A=I_{2}$. Note that $I_{1}$ is a Morley
sequence of $p$ over $I_{2}$. Also $\cut c$ is interior to $J_{2}$,
and $J_{4}$ is a Morley sequence of $p$ over $J_{\neq4}I_{2}c$
by the choice of $\left(j_{1},j_{2}\right)$. We already mentioned
that $J_{2}\cup\left\{ \bar{a}\right\} $ is indiscernible over $J_{\neq2}I_{2}c$.
Finally, the conclusion of Fact \ref{fact:extenal domination} is
exactly what we want.
\end{proof}

\subsubsection*{Right to Left}

Assume that $T$ is NIP but not distal, and that $\left|T\right|<\kappa$
is singular such that $\kappa^{+}=2^{\kappa}$. 

For an ordinal $\alpha<\kappa$, let $\left(Y_{\alpha},<\right)$
be a dense linear order of cofinality $\left|\alpha\right|^{+}$ and
power $\left|\alpha\right|^{+}$. let $\left(X_{\alpha},<\right)$
be the linear order $Y_{\alpha}+Y_{\alpha}^{\opp}$. Let $\cut c_{\alpha}$
be the cut $\left(Y_{\alpha},Y_{\alpha}^{\opp}\right)$. 

Since $T$ is not distal, there is an indiscernible sequence $\OurSequence$
which is not distal. By Remark \ref{rem:about distal indiscernible sequences},
we may assume that $\OurSequence$ is dense, and of order type $\sum_{\alpha<\kappa}X_{\alpha}$
(see Notation \ref{notation: sum of orders}). Abusing notation, we
let $\cut c_{\alpha}$ denote the appropriate cuts in $\OurSequence$.
Note that $\left|\OurSequence\right|=\kappa$. 
\begin{defn}
Let $D=D_{\OurSequence}$. 
\end{defn}
Note that by Shrinking (Fact \ref{fact:(Shrinking)}), given some
set $A$ of size $<\kappa$, any cut $\cut c$ in $\OurSequence$
which is not one of the cuts\footnote{Here we use the term ``cuts'' in the most general sense, as opposed
to our convention so far where cuts had infinite cofinality from both
sides.} induced by the finite equivalence relations induced by $A$ (there
are at most $\left|T\right|+\left|A\right|$ such) which has cofinality
$\left(\left|T\right|+\left|A\right|\right)^{+}$ from both sides
is $A$-\generic{}. In particular, given such an $A$, for some $\alpha<\kappa$,
$\cut c_{\alpha}$ is \generic{} for $A$. 
\begin{lem}
\label{lem:Making one step}Suppose $A$ is a $D$-set, $\left|A\right|<\kappa$.
Suppose that $p\left(x\right)\in S\left(A\right)$. Then there is
some $q\supseteq p$, $q\in S\left(AI_{0}\right)$ where $I_{0}\subseteq\OurSequence$,
$\left|I_{0}\right|\leq\left|T\right|$ such that if $b\models q$
then $bA$ is a $D$-set. \end{lem}
\begin{proof}
Try to construct an increasing continuous sequence of partial types
$\sequence{p_{\varepsilon}}{\varepsilon<\left|T\right|^{+}}$ and
intervals $I_{\varepsilon}=\left(i_{1}^{\varepsilon},i_{2}^{\varepsilon}\right)\subseteq X_{\alpha_{\varepsilon}}$
around distinct $A$-\generic{} cuts $\cut c_{\alpha_{\varepsilon}}$
in $\OurSequence$ such that $p_{0}=p$, $p_{\varepsilon+1}\backslash p_{\varepsilon}$
contains one formula over $AI_{\varepsilon}$. Also, we ask that each
$I_{\varepsilon}$ is indiscernible over $A\OurSequence\backslash I_{\varepsilon}$
(in other words, they witness that $\cut c_{\alpha_{\varepsilon}}$
are \generic{} for $A$). 

Suppose $p_{\varepsilon}$ (or any completion of it) is not as we
wanted. This means that there is some $b\models p_{\varepsilon}$
such that $bA$ is not a $D$-set. Let $\alpha_{\varepsilon}<\kappa$
be such that $\cut c_{\alpha_{\varepsilon}}$ is a \generic{} cut
for $bA$ and that $\alpha_{\varepsilon}\notin\set{\alpha_{\zeta}}{\zeta<\varepsilon}$,
and suppose this is witnessed by $I_{\varepsilon}=\left(i_{1}^{\varepsilon},i_{2}^{\varepsilon}\right)\subseteq X_{\alpha_{\varepsilon}}$.
By definition, some finite tuple $ba$ from $bA$ is not a $D$-set,
which means that $\tp\left(ba/\OurSequence\right)$ is not orthogonal
to $\lim\left(\cut c_{\alpha_{\varepsilon}}/\OurSequence\right)$.
This means that there is some $c$ filling $\cut c_{\alpha_{\varepsilon}}$
such that $c$ does not fill $\cut c_{\alpha_{\varepsilon}}$ in $I_{\varepsilon}$
over $ba\cup\OurSequence\backslash I_{\varepsilon}$. 

Hence for some formula $\varphi\left(y_{0},\ldots,y_{i-1},y,y_{i+1},\ldots,y_{k},x,w\right)$
over $\OurSequence\backslash I_{\varepsilon}$ and some $c_{0}<\ldots<c_{i-1}<\cut c_{\alpha_{\varepsilon}}<c_{i+1}<\ldots<c_{k}$
from $I_{\varepsilon}$ such that $\neg\varphi\left(c_{0},\ldots,c_{i-1},c,c_{i+1},\ldots,c_{k},b,a\right)$
holds while for all $d_{0}<\ldots<d_{k}<\cut c_{\alpha}$ from $I_{\varepsilon}$,
$\varphi\left(d_{0},\ldots,d_{k},b,a\right)$ holds. 

As $A$ is a $D$-set, by Main Lemma \ref{lem:Main Lemma simple},
\emph{$c$ fills $\cut c_{\alpha_{\varepsilon}}$ in $I_{\varepsilon}$
over $A\cup\OurSequence\backslash I_{\varepsilon}$}. 

Let $c'<c''\in I_{\varepsilon}$ be such that $c_{i-1}<c'<c''<\cut c_{\alpha}$.
Then for any formula $\theta\in p_{\varepsilon}$,

\[
\C\models\exists x\theta\left(x\right)\land\neg\varphi\left(c_{0},\ldots,c_{i-1},c,c_{i+1},\ldots,c_{k},x,a\right)\land\varphi\left(c_{0},\ldots,c_{i-1},c',c_{i+1},\ldots,c_{k},x,a\right),
\]
 and hence
\[
\C\models\exists x\theta\left(x\right)\land\neg\varphi\left(c_{0},\ldots,c_{i-1},c'',c_{i+1},\ldots,c_{k},x,a\right)\land\varphi\left(c_{0},\ldots,c_{i-1},c',c_{i+1},\ldots,c_{k},x,a\right).
\]

Put 
\[
p_{\varepsilon+1}=p_{\varepsilon}\cup\left\{ \neg\varphi\left(c_{0},\ldots,c_{i-1},c'',c_{i+1},\ldots,c_{k},x,a\right)\land\varphi\left(c_{0},\ldots,c_{i-1},c',c_{i+1},\ldots,c_{k},x,a\right)\right\} .
\]
Then $p_{\varepsilon+1}$ is consistent. 

Finally, let $p_{\infty}=\bigcup_{\varepsilon<\left|T\right|^{+}}p_{\varepsilon}$,
and let $b\models p_{\infty}$. Then we have a sequence of mutually
indiscernible sequences $\sequence{I_{\varepsilon}}{\varepsilon<\left|T\right|^{+}}$
over $A$ and some $b$ such that for each $\varepsilon<\left|T\right|^{+}$,
$I_{\varepsilon}$ is not indiscernible over $b$. This means that
the dp-rank of $\tp\left(b/A\right)$ is $\geq\left|T\right|^{+}$,
which implies the independence property. See for example \cite[Corollary 2.3]{KaplanSimon}.

Hence we must get stuck somewhere, and we are done. \end{proof}
\begin{cor}
\label{cor:Main corollary}Suppose $A\subseteq B$ are $D$-sets and
$\left|A\right|<\kappa$, $\left|B\right|\leq\kappa$. Assume that
$p\left(x\right)\in S\left(A\right)$, then there is some $b\models p$
such that $Bb$ is a $D$-set.\end{cor}
\begin{proof}
Write $B$ as an increasing continuous sequence $\bigcup_{\alpha<\kappa}B_{\alpha}$
where $\left|B_{\alpha}\right|<\kappa$ and $A=B_{0}$. Construct
an increasing continuous sequence of types $\sequence{q_{\alpha}}{\alpha<\kappa}$
and subsets $\sequence{I_{\alpha}}{\alpha<\kappa}$ of $\OurSequence$
such that $p=q_{0}$, $q_{\alpha+1}\in S\left(B_{\alpha+1}\cup I_{\alpha+1}\right)$,
$\left|I_{\alpha}\right|<\kappa$ for all $\alpha<\kappa$ and if
$b\models q_{\alpha+1}$ then $bB_{\alpha+1}$ is a $D$-set. 

For $\alpha=0$ and limit there is nothing to do. For $\alpha+1$,
first choose $q_{\alpha+1}'\in S\left(B_{\alpha+1}\cup I_{\alpha}\right)$
extending $q_{\alpha}$. Apply Lemma \ref{lem:Making one step} to
get some $J$ of size $\leq\left|T\right|$ and a type $q_{\alpha+1}\in S\left(B_{\alpha+1}\cup I_{\alpha}J\right)$
such that if $b\models q_{\alpha+1}$ then $bB_{\alpha+1}$ is a $D$-set
(which is the same as saying that $bB_{\alpha+1}\OurSequence$ is
a $D$-set, see Remark \ref{rem:reverse}). Finally, let $I_{\alpha+1}=I_{\alpha}J$. 

When the construction is done, let $q=\bigcup_{\alpha<\kappa}q_{\alpha}$
and $b\models q$. 
\end{proof}
We can finally proof the right to left direction of Theorem \ref{thm:Distal-exact saturation}.
We wish to construct a $\kappa$-saturated model $M$ which is not
$\kappa^{+}$-saturated. 
\begin{prop}
\label{prop:not kappa^+ saturated}Any $D$-model $M\supseteq\OurSequence$
is not $\kappa^{+}$-saturated.\end{prop}
\begin{proof}
As $\OurSequence$ is not distal, the limit type of any cut $\cut c$
in it is not orthogonal to any other limit cut. See Remark \ref{rem:about distal indiscernible sequences}.
This means that the type $\lim\left(\cut c/\OurSequence\right)$ is
not a $D$-type, and hence not realized in $M$. 
\end{proof}

\begin{proof}
[Proof of Theorem \ref{thm:Distal-exact saturation}]By Proposition
\ref{prop:not kappa^+ saturated}, it is enough to construct a $\kappa$-saturated
$D$-model containing $\OurSequence$. We do this in similar way to
the one in the proof of Theorem \ref{thm:Main-Simple theories}. 

Let $\sequence{S_{\alpha}}{\alpha<\kappa^{+}}$ be a partition of
$\kappa^{+}$ to sets of size $\kappa^{+}$. Construct an increasing
continuous sequence of $D$-sets $\sequence{A_{\alpha}}{\alpha<\kappa^{+}}$
and sequences of types $\sequence{\bar{p}_{\alpha}}{\alpha<\kappa^{+}}$
such that:
\begin{enumerate}
\item $\left|A_{\alpha}\right|\leq\kappa$;
\item $\bar{p}_{\alpha}$ is an enumeration $\sequence{p_{\alpha,\beta}}{\beta\in S_{\alpha}\backslash\alpha}$
of all complete types over subsets of $A_{\alpha}$ of size $<\kappa$
(this uses $\kappa^{+}=2^{\kappa}$); 
\item If $\alpha\in S_{\gamma}$ and $\gamma\leq\alpha$, then $A_{\alpha+1}$
contains a realization of $p_{\gamma,\alpha}$.
\end{enumerate}
Start with $A_{0}=\OurSequence$, see Remark \ref{rem:reverse}. Step
(3) is done by Corollary \ref{cor:Main corollary}. Finally, let $M=\bigcup_{\alpha<\kappa^{+}}A_{\alpha}$
and we are done. 
\end{proof}
\bibliographystyle{alpha}
\bibliography{common2}

\end{document}